\documentclass[12pt]{article}
\usepackage{amsthm,amsmath,amssymb,enumerate}

\numberwithin{equation}{section}

\theoremstyle{plain}
\newtheorem{thm}{Theorem}[section]

\newtheorem{lem}[thm]{Lemma}
\theoremstyle{definition}
\newtheorem*{rem}{Remark}

\newcommand{\Z}{\mathbb{Z}}
\newcommand{\F}{\mathbb{F}}

\newcommand{\bs}{\boldsymbol}
\newcommand{\aut}{\operatorname{Aut}}
\newcommand{\GL}{\operatorname{GL}}
\newcommand{\sym}{\operatorname{Sym}}

\begin{document}
\title{On a huge family of non-schurian Schur rings}

\author{Akihide Hanaki\thanks{Supported by JSPS KAKENHI Grant Number JP17K05165.}\\
\small Faculty of Science, Shinshu University,\\[-0.8ex]
\small Matsumoto, 390-8621, Japan\\[-0.8ex] 
\small\tt hanaki@shinshu-u.ac.jp\\
\and
Takuto Hirai \\
\small Department of Science,\\[-0.8ex]
\small Graduate School of Science and Technology,\\[-0.8ex]
\small Shinshu University,\\[-0.8ex]
\small Matsumoto, 390-8621, Japan\\[-0.8ex]
\small\tt takuto.hirai2021@gmail.com\\
\and
Ilia Ponomarenko \\
\small St.~Petersburg Department of the Steklov Mathematical Institute,\\[-0.8ex]
\small St.~Petersburg, Russia\\[-0.8ex]
\small\tt inp@pdmi.ras.ru}

%\keywords{Schur ring, permutation group}
%\subjclass[2010]{20C05}

\date{}

\maketitle

\begin{abstract}
  In his famous monograph on permutation groups, H.~Wielandt gives an example of a Schur ring
  over an elementary abelian group of order $p^2$ ($p>3$ is a prime),
  which is non-schurian, that is, it is the transitivity module of no  permutation group.
  Generalizing this example, we construct a huge family of non-schurian Schur rings
  over elementary abelian groups of even rank.
\end{abstract}

\section{Introduction}
A {\it Schur ring} over a finite group $H$  is a subring of a group algebra of $H$,
which has a distinguished linear basis corresponding to a certain partition of $H$.
A typical example of a Schur ring is obtained when $H$ is a regular subgroup of
a group $G\le\sym(H)$ and the partition is formed by the orbits of the stabilizer of $1_H$ in $G$.
These rings were introduced by I.~Schur (1933) and named after him {\it schurian}.
However there are non-schurian Schur rings.
Apparently, the first such example was given by H.~Wielandt \cite[Theorem 26.4]{MR0183775}
for elementary abelian groups $H$ of rank $2$; some other examples can be found in \cite{MR1850191,MR3225122}.
A goal of the present paper is to generalize Wielandt's example by constructing non-schurian Schur rings
over all elementary abelian groups of even rank except for the order $2^2$, $3^2$, and $2^4$.

Let $\F=\F_q$ be a Galois field of order $q$. 
We denote by  $\mathcal{L}$ the set of all lines, $1$-dimensional subspaces,
in the $2$-dimensional $\F$-vector space $V=\F^2$.
Then $\mathcal{L}=\{L_\alpha\mid \alpha\in \F\cup\{\infty\}\}$, where
\begin{eqnarray*}
  L_\alpha &=& \{(x,\alpha x)\mid x\in \F\}\quad  (\alpha\in \F),\\
  L_\infty &=& \{(0,x)\mid x\in \F\}.
\end{eqnarray*}

Let $\Pi=\{P_1,\dots,P_r\}$ be a partition of $\mathcal{L}$ into a disjoint union of $r\ge 1$ subsets $P_i$, $i=1,\ldots,r$.
This partition induces a partition $\widetilde{\Pi}=\{\{(0,0)\}, \widetilde{P}_1,\dots,\widetilde{P}_r\}$
of the vector space $V$, where $\widetilde{P}_i=\bigcup_{L_\alpha\in P_i}L_\alpha^\sharp$ and
$L_\alpha^\sharp=L_\alpha\setminus\{(0,0)\}$.
It is easy to see that the partition $\widetilde{\Pi}$ defines
a Schur ring $\mathcal{S}(\Pi)$ over the additive group $V^+$
which is an elementary abelian group of order $q^2$ (Theorem \ref{thm:Schur}).
We set 
\begin{equation}\label{def_of_M}
\mathcal{M}(\Pi)=\{\alpha\in \F\cup\{\infty\}\mid \{L_\alpha\}\in \Pi\}.
\end{equation}

Our main theorem shows that,
if $\{\infty, 0, 1\}\subset \mathcal{M}(\Pi)$ and $\mathcal{M}(\Pi)\setminus\{\infty\}$
is not a subfield of $\F$,
then the Schur ring $\mathcal{S}(\Pi)$ is not schurian (Theorem \ref{thm:nonschurian}).
The above mentioned Wielandt's example is just the case when $q\ge 5$ is a prime number and
$\Pi=\{\{L_\infty\}, \{L_0\}, \{L_1\}, \mathcal{L}\setminus\{L_\infty,L_0,L_1\}\}$.

\section{Proofs of the main results}

\subsection{Schur rings}
Let $H$ be a finite group, $\Z H$ the group ring of $H$ over the ring of rational integers $\Z$,
and $\Pi$ a partition  of $H$. Set
$$
\mathcal{A}=\bigoplus_{X\in\Pi} \Z \underline{X}\subset \Z H,
$$
where for any set $X\subset H$, we put $\underline{X}=\sum_{x\in X}x\in \Z H$.
Following \cite{MR0183775}, we say that $\mathcal{A}$ is a \emph{Schur ring}
over $H$ if the following conditions are satisfied:
\begin{enumerate}
  \item[(S1)] $\{1_H\}\in \Pi$,
  \item[(S2)] $\{x^{-1}\mid x\in X\}\in \Pi$ for all $X\in\Pi$, and
  \item[(S3)] $\mathcal{A}$ is a subring of $\Z H$.
\end{enumerate}

A typical example of a Schur ring is obtained as follows.
Let $G\le\sym(H)$ be a (transitive) permutation group containing $H$ as a regular subgroup,
and let $G_1$ be the stabilizer of the point $1_H$ in $G$.
Then the partition of $H$ into the  $G_1$-orbits defines a Schur ring over $H$ \cite[Theorem 24.1]{MR0183775}.
Any  Schur ring obtained in this way is said to be \emph{schurian}.

For more details on Schur rings the reader is referred to \cite{MR2535399}.

\subsection{Construction}
Keeping the notations from Introduction,
let $q$ be a power of prime and $H$ an elementary abelian group of order $q^2$.
To avoid misunderstanding, we use multiplicative notation for $H$ and fix an isomorphism $\rho:V^+\to H$.
Let $\Pi=\{P_1,\dots,P_r\}$ be a partition of $\mathcal{L}$,
and let $\widetilde{P}_1,\dots,\widetilde{P}_r$ be as in Introduction. Set
$$
\underline{\widetilde{P}_i}=\sum_{L_\alpha\in P_i}\sum_{\bs{x}\in L_\alpha^\sharp}\rho(\bs{x})\in \Z H,\quad i=1,\ldots,r,
$$
and
$$
\mathcal{S}(\Pi)=\Z\rho(0,0)\oplus \left(\bigoplus_{i=1}^r \Z \underline{\widetilde{P}_i}\right) \subset \Z H.
$$
% It is easily seen that if 
% $$
% Q_i=\sum_{\alpha\in P_i}\sum_{\bs{x}\in L_\alpha}\rho(\bs{x})\in \Z H,\quad i=1,\ldots,r,
% $$
% then 
% \begin{equation}\label{Schur.gen}
%   \mathcal{S}(\Pi)=\Z\rho(0,0)\oplus\left(\bigoplus_{i=1}^r \Z Q_i\right).
% \end{equation}

\begin{thm}\label{thm:Schur}
  Let $\Pi$ be an arbitrary partition of $\mathcal{L}$.
  Then $\mathcal{S}(\Pi)$ is a Schur ring over the group $H$.
\end{thm}

\begin{proof}
  The conditions (S1) and (S2) are clear by definition.
  It is easily seen that if 
  $$
  Q_i=\sum_{L_\alpha\in P_i}\sum_{\bs{x}\in L_\alpha}\rho(\bs{x})\in \Z H,\quad i=1,\ldots,r,
  $$
  then 
  \begin{equation}\label{Schur.gen}
    \mathcal{S}(\Pi)=\Z\rho(0,0)\oplus\left(\bigoplus_{i=1}^r \Z Q_i\right).
  \end{equation}
  Since $\mathcal{S}(\Pi)$ is closed with respect to addition, it suffices to show
  that $\mathcal{S}(\Pi)$ written in form (\ref{Schur.gen})
  is closed with respect to multiplication.
  Let $\alpha$ and $\beta$ be distinct elements of $\F\cup\{\infty\}$.
  Then
  $$
  \left(\sum_{\bs{x}\in L_\alpha}\rho(\bs{x})\right)\left(\sum_{\bs{y}\in L_\beta}\rho(\bs{y})\right)
  =\sum_{\bs{x}\in L_\alpha}\sum_{\bs{y}\in L_\beta}\rho(\bs{x}+\bs{y})=\sum_{\bs{z}\in V}\bs{z}
  \in \mathcal{S}(\Pi),
  $$
  because $L_\alpha$ and $L_\beta$ are distinct $1$-dimensional subspaces of the $2$-dimensional vector space $V$.
  Also we have
  $$\left(\sum_{\bs{x}\in L_\alpha}\rho(\bs{x})\right)^2=q\left(\sum_{\bs{x}\in L_\alpha}\rho(\bs{x})\right).$$
  Thus, if $i\ne j$, then
  $$
  Q_i Q_j=|P_i|\,|P_j| \sum_{\bs{z}\in V}\bs{z}   \in \mathcal{S}(\Pi),
  $$
  and
  $$
  {Q_i}^2=q Q_i + |P_i|(|P_i|-1) \sum_{\bs{z}\in V}\bs{z}   \in \mathcal{S}(\Pi).
  $$
  Thus, $\mathcal{S}(\Pi)$ is a Schur ring over $H$.
\end{proof}

\begin{rem}
  It is well known that every Schur ring defines an association scheme.
  Moreover, let $\Pi$ be the partition of $\mathcal{L}$ into singletons  $\{L_\alpha\}$, $\alpha\in \F\cup\{\infty\}$.
  Then the association scheme corresponding to the Schur ring $\mathcal{S}(\Pi)$ is
  a cyclotomic scheme over $\F_{q^2}$.
  Such a scheme is known to be amorphic \cite{MR2557883}.
  Thus  Theorem \ref{thm:Schur} is a direct consequence of this fact.
\end{rem}

\subsection{The main theorem}
We are ready to state the main result of the present paper.

\begin{thm}\label{thm:nonschurian}
  Let $\F$ be a Galois field of order $q$,
  $\mathcal{L}$ the set of all lines in the vector space $V=\F^2$,
  $\Pi$ a partition of $\mathcal{L}$, and $\mathcal{S}(\Pi)$ the Schur ring
  over the elementary abelian group $H\cong V^+$ of order $q^2$.  Suppose that 
  \begin{equation}\label{condition}
    \{\infty, 0, 1\}\subset \mathcal{M}(\Pi)\quad\text{and}\quad \mathcal{M}(\Pi)\setminus\{\infty\}\text{ 
      is not a subfield of }\F,
  \end{equation}
  where $\mathcal{M}(\Pi)$ is defined by \eqref{def_of_M}.
  Then the Schur ring $\mathcal{S}(\Pi)$  is non-schurian.
\end{thm}

\begin{rem}
  The first assumption in \eqref{condition} can be replaced by the assumption $|\mathcal{M}(\Pi)|\ge 3$.
  This follows from the fact that the action of $\GL(2,q)$ on $\mathcal{L}$ is $3$-transitive for $q>3$
  \cite[p.~245]{MR1409812}.
  Note that the second assumption in \eqref{condition} is invariant
  with respect to  the action of $\GL(2,q)$ on $\mathcal{L}$
  under the assumption $\{\infty, 0, 1\}\subset \mathcal{M}(\Pi)$.
\end{rem}

To prove Theorem \ref{thm:nonschurian}, let $q=p^e$,
where $p$ is a prime and $e\ge 1$ is an integer.
We need the following auxiliary lemma.

\begin{lem}\label{lem01}
  Let $\sigma$ be an $\F_p$-linear transformation on $V$
  such that the sets $L_0$, $L_1$, and $L_\infty$  are $\sigma$-invariant.
  Let $\alpha,\beta\in \F$ be such that the sets $L_\alpha$ and $L_\beta$ are $\sigma$-invariant.
  Then so are $L_{\alpha+\beta}$ and $L_{\alpha\beta}$.
\end{lem}

\begin{proof}
  We fix an element $\zeta\in \F$ such that $\F=\F_p[\zeta]$;
  for example, $\zeta$ is a primitive $(p^e-1)$th root of unity.
  Then the set $\{1,\zeta,\dots,\zeta^{e-1}\}$ is an $\F_p$-basis of $\F$.
  Given $x\in \F$ or $x\in V$, we define the column vector
  $$
  V(x)=[x,\zeta x,\dots,\zeta^{e-1} x]^T.
  $$
  Let $M(e,p)$ be the full matrix algebra of degree $e$ over $\F_p$, and
  let $\Psi:\F\to M(e,p)$ be the regular representation of $\F$ as an $\F_p$-algebra
  with respect to the basis $\{1,\zeta,\dots,\zeta^{e-1}\}$.
  Namely, $ \Psi(\gamma)V(1)=V(\gamma)$ for all $\gamma\in \F$.

  Let $\bs{x}=(1,0)$ and $\bs{y}=(0,1)\in V$. Then
  $$
  \{\bs{x},\zeta\bs{x},\dots,\zeta^{e-1}\bs{x}\},\ 
  \{\bs{y},\zeta\bs{y},\dots,\zeta^{e-1}\bs{y}\},\ 
  \{(\bs{x}+\bs{y}),\zeta(\bs{x}+\bs{y}),\dots,\zeta^{e-1}(\bs{x}+\bs{y})\}
  $$
  are $\F_p$-bases of $L_0$, $L_\infty$, and $L_1$, respectively.
  Since $L_0$, $L_1$, and $L_\infty$  are $\sigma$-invariant,
  there exist matrices $A, B, C\in M(e,p)$ such that
  $$
  \sigma(\bs{s}V(\bs{x}))=\bs{s}AV(\bs{x}),\ 
  \sigma(\bs{s}V(\bs{y}))=\bs{s}BV(\bs{y}),\ 
  \sigma(\bs{s}V(\bs{x}+\bs{y}))=\bs{s}CV(\bs{x}+\bs{y})$$
  for all row vectors $\bs{s}\in (\F_p)^e$.
  Since the set $\{\bs{x},\zeta\bs{x},\dots,\zeta^{e-1}\bs{x},\bs{y},\zeta\bs{y},\dots,\zeta^{e-1}\bs{y}\}$
  is $\F_p$-linearly independent, we have $A=B=C$.

  By hypothesis, $L_\alpha$ is $\sigma$-invariant.
  Therefore by the above argument for $\F_p$-bases
  $\{\bs{x},\zeta\bs{x},\dots,\zeta^{e-1}\bs{x}\}$,
  $\{\alpha\bs{y},\zeta\alpha\bs{y},\dots,\zeta^{e-1}\alpha\bs{y}\}$,  
  $\{(\bs{x}+\alpha\bs{y}),\zeta(\bs{x}+\alpha\bs{y}),\dots,\zeta^{e-1}(\bs{x}+\alpha\bs{y})\}$
  of $L_0$, $L_\infty$, and $L_\alpha$, respectively, we have
  \begin{equation}\label{eq:linmap}
  \sigma(\bs{s}V(\alpha\bs{y}))=\bs{s}AV(\alpha\bs{y})=\bs{s}A\Psi(\alpha)V(\bs{y})
  \end{equation}
  for all $\bs{s}\in (\F_p)^e$. On the other hand,
  $$
  \sigma(\bs{s}V(\alpha\bs{y}))=\sigma(\bs{s}\Psi(\alpha)V(\bs{y}))=\bs{s}\Psi(\alpha)AV(\bs{y}).$$
  Thus, $A\Psi(\alpha)=\Psi(\alpha)A$ and similarly $A\Psi(\beta)=\Psi(\beta)A$.

  The set $\{\bs{x}+\alpha\bs{y}+\beta\bs{y}, \zeta(\bs{x}+\alpha\bs{y}+\beta\bs{y}),
  \dots,\zeta^{e-1}(\bs{x}+\alpha\bs{y}+\beta\bs{y})\}$ is  a basis of $L_{\alpha+\beta}$.
  Since also
  \begin{eqnarray*}
    \sigma(\bs{s}V(\bs{x}+\alpha\bs{y}+\beta\bs{y}))
    &=& \sigma(\bs{s}V(\bs{x}))+\sigma(\bs{s}V(\alpha\bs{y}))+\sigma(\bs{s}V(\beta\bs{y}))\\
    &=& \bs{s}AV(\bs{x})+\bs{s}AV(\alpha\bs{y})+\bs{s}AV(\beta\bs{y})\\
    &=& \bs{s}AV(\bs{x}+\alpha\bs{y}+\beta\bs{y})\in L_{\alpha+\beta}
  \end{eqnarray*}
  by (\ref{eq:linmap}), the set $L_{\alpha+\beta}$ is $\sigma$-invariant.
  
  The set $\{\bs{x}+\alpha\beta\bs{y}, \zeta(\bs{x}+\alpha\beta\bs{y}),
  \dots,\zeta^{e-1}(\bs{x}+\alpha\beta\bs{y})\}$ is a basis of $L_{\alpha\beta}$. Since also
  \begin{eqnarray*}
    \sigma(\bs{s}V(\bs{x}+\alpha\beta\bs{y}))
    &=& \sigma(\bs{s}V(\bs{x}))+\sigma(\bs{s}V(\alpha\beta\bs{y}))\\
    &=& \sigma(\bs{s}V(\bs{x}))+\sigma(\bs{s}\Psi(\alpha)\Psi(\beta)V(\bs{y}))\\
    &=& \bs{s}AV(\bs{x})+\bs{s}\Psi(\alpha)\Psi(\beta)AV(\bs{y})\\
    &=& \bs{s}AV(\bs{x})+\bs{s}A\Psi(\alpha)\Psi(\beta)V(\bs{y})\\
    &=& \bs{s}AV(\bs{x})+\bs{s}AV(\alpha\beta\bs{y})
        = \bs{s}AV(\bs{x}+\alpha\beta\bs{y})\in L_{\alpha\beta},
  \end{eqnarray*} 
  the set  $L_{\alpha\beta}$ is $\sigma$-invariant.
\end{proof}

\begin{proof}[Proof of Theorem {\rm\ref{thm:nonschurian}}]
  Suppose that the Schur ring $\mathcal{S}(\Pi)$ is schurian.
  Then there is a (transitive) permutation group  $G\le\sym(H)$ containing $H=\rho(V^+)$ as a regular subgroup. Moreover,
  $$
  H=\rho(L_0)\times \rho(L_1)\quad\text{and}\quad\underline{\rho(L_0)}, \underline{\rho(L_1)}\in \mathcal{S}(\Pi).
  $$
  Applying \cite[Theorem 26.2]{MR0183775} and the first part of the proof of \cite[Lemma 26.3]{MR0183775},
  we  conclude that $H$ is normal in $G$.
  According to \cite[Theorem 4.2]{MR2963408},
  we may assume that the stabilizer $G_1$ is a subgroup of $\aut(H)\cong \GL(2e,p)$.

  Now let $\alpha,\beta\in \mathcal{M}(\Pi)\setminus\{\infty\}$.
  Then given $\sigma\in G_1$, the sets $L_\alpha$ and $L_\beta$ are obviously $\sigma$-invariant.
  By Lemma \ref{lem01}, this implies that so are the sets $L_{\alpha+\beta}$ and $L_{\alpha\beta}$.
  It follows that 
  $$
  \underline{\rho(L_{\alpha+\beta})}, \underline{\rho(L_{\alpha\beta})}\in \mathcal{S}(\Pi)\quad\text{and}\quad
  \alpha+\beta,\ \alpha\beta\in \mathcal{M}(\Pi).
  $$
  Thus, $\mathcal{M}(\Pi)\setminus\{\infty\}$ must be a subfield of $\F$, which completes the proof.
\end{proof}

% \section*{Acknowledgments}

\bibliographystyle{amsplain}
\providecommand{\bysame}{\leavevmode\hbox to3em{\hrulefill}\thinspace}
\providecommand{\MR}{\relax\ifhmode\unskip\space\fi MR }
\providecommand{\MRhref}[2]{%
  \href{http://www.ams.org/mathscinet-getitem?mr=#1}{#2}
}
\providecommand{\href}[2]{#2}

\end{document}